\documentclass[submission]{FPSAC2022}

\usepackage{tikz-cd}

\newtheorem{theorem}{Theorem}

\newtheorem{defn}{Definition}
\newtheorem{lem}{Lemma}
\newtheorem{proposition}{Proposition}
\newtheorem{corollary}{Corollary}
\newtheorem{example}{Example}

\newtheorem{remark}{Remark}

\newcommand{\al}{\alpha}
\newcommand{\be}{\beta}

\newcommand{\RR}{\mathbb{R}} 
\newcommand{\PP}{\mathbb{P}} 
\newcommand{\kk}{\mathbf{k}} 
\newcommand{\QSym}{\operatorname{QSym}}
\newcommand{\Des}{\operatorname{Des}}

\newcommand{\Peak}{\operatorname{Peak}}
\newcommand{\set}[1]{\left\{ #1 \right\}}

\newcommand{\tup}[1]{\left( #1 \right)}

\makeatletter
\providecommand*{\shuffle}{%
  \mathbin{\mathpalette\shuffle@{}}%
}
\newcommand*{\shuffle@}[2]{%
  \sbox0{$#1\vcenter{}$}%
  \kern .15\ht0 
  \rlap{\vrule height .25\ht0 depth 0pt width 2.5\ht0}%
  \raise.1\ht0\hbox to 2.5\ht0{%
    \vrule height 1.75\ht0 depth -.1\ht0 width .17\ht0 %
    \hfill
    \vrule height 1.75\ht0 depth -.1\ht0 width .17\ht0 %
    \hfill
    \vrule height 1.75\ht0 depth -.1\ht0 width .17\ht0 %
  }%
  \kern .15\ht0 
}
\makeatother

\title{A $q$-deformation of enriched $P$-partitions}

\author[D. Grinberg \and E.A. Vassilieva]{Darij Grinberg\thanks{\href{mailto:darijgrinberg@gmail.com}{darijgrinberg@gmail.com}}\addressmark{1}, \and Ekaterina A. Vassilieva\thanks{\href{mailto:katya@lix.polytechnique.fr}{katya@lix.polytechnique.fr}}\addressmark{2}}

\address{\addressmark{1}Department of Mathematics, Drexel University, Philadelphia, PA 19104, USA\\ \addressmark{2}Laboratoire d'Informatique de l'Ecole Polytechnique, Palaiseau, France}

\received{\today}

\abstract{We introduce a $q$-deformation that generalises in a single framework previous works on classical and enriched $P$-partitions. In particular, we build a new family of power series with a parameter $q$ that interpolates between Gessel's fundamental ($q=0$) and Stembridge's peak quasisymmetric functions ($q=1$) and show that it is a basis of $\QSym$ when $q\notin\{-1,1\}$. Furthermore we build their corresponding monomial bases parametrised with $q$ that cover our previous work on enriched monomials and the essential quasisymmetric functions of Hoffman.}

\resume{Nous introduisons une $q$-déformation qui généralise dans un cadre unique les travaux antérieurs sur les $P$-partitions classiques et enrichies. En particulier, nous construisons une famille de séries formelles avec un paramètre $q$ qui interpole entre les fonctions quasisymétriques fondamentales de Gessel ($q=0$) et les fonctions de pic de Stembridge ($q=1$) et montrons qu'il s'agit d'une base de $\QSym$ quand $q\notin\{-1,1\}$. De plus, nous construisons leur bases de monômes associées paramétrées par $q$ qui généralisent nos travaux sur les monômes enrichis et les fonctions essentielles de Hoffman.}


\usepackage[backend=bibtex]{biblatex}
\begin{document}

\maketitle
\section{Introduction}
Introduced by Stanley in \cite{Sta72}, $P$-partitions are order preserving maps from a partially ordered set $P$ to the set of positive integers with many significant applications in algebraic combinatorics. In particular, they are the building block of Gessel's ring of quasisymmetric functions ($\QSym$) in  \cite{Ges84}. Replacing positive integers by signed ones, Stembridge introduces in \cite{Ste97} an enriched version of $P$-partitions to build the algebra of peaks, a  subalgebra of $\QSym$. The generating functions of classical (enriched) $P$-partitions on labelled chains are the fundamental (peak) quasisymmetric functions, an important basis of $\QSym$ (the algebra of peaks) related to the descent (peak) statistic on permutations. More recently, in \cite{GriVas21}, we redefine these generating functions on weighted posets to extend their nice properties to the monomial and enriched monomial bases of $\QSym$. However the classical and enriched frameworks remained so far separated. We merge them into one via a new $q$-deformation of the generating function for enriched $P$-partitions that interpolates between Gessel's and Stembridge's works.   
\subsection{Posets and enriched $P$-partitions}
\label{sec : poset}
We recall the main definitions regarding posets and (enriched) $P$-partitions. The reader is referred to \cite{Sta01, Ges84, Ste97} for further details.   
\begin{defn}[Labelled posets] Let $[n] = \left\{1,2,\dots, n\right\}$. A \emph{labelled poset} $P=([n],<_P)$ is an arbitrary partial order $<_P$ on the set $[n]$. 
\end{defn}
\begin{defn}[$P$-partition]\label{def : ppart}
Let $\PP=\left\{1,2,3,\dots\right\}$ and let $P = ([n],<_P)$ be a labelled poset.
A \emph{$P$-partition} is a map $f: [n]\longrightarrow \PP$ that satisfies the two following conditions:
\begin{itemize}
\item[(i)] If $i <_P j$, then $f(i) \leq f(j)$.
\item[(ii)] If $i <_P j$ and $i > j$, then $f(i) < f(j)$.
\end{itemize}
The relations $<$ and $>$ stand for the classical order on $\PP$. Let $\mathcal{L}_\PP(P)$ denote the set of $P$-partitions.
\end{defn}
\begin{defn}[Enriched $P$-partition]\label{def : enriched}
Let $\PP^{\pm}$ be the set of positive and negative integers totally ordered by $-1<1<-2<2<-3<3<\dots$. We embed $\PP$ into $\PP^{\pm}$ and let $-\PP \subseteq \PP^{\pm}$ be the set of all $-n$ for $n \in \PP$. Given a labelled poset $P = ([n],<_P)$, an \emph{enriched $P$-partition} is a map $f: [n]\longrightarrow \PP^{\pm} $ that satisfies the two following conditions:
\begin{itemize}
\item[(i)] If $i <_P j$ and $i < j$, then $f(i) < f(j)$ or $f(i) = f(j) \in \PP$.
\item[(ii)] If $i <_P j$ and $i>j$, then $f(i) < f(j)$ or $f(i) = f(j) \in -\PP$.
\end{itemize}
Further, let $\mathcal{L}_{\PP^{\pm}}(P)$ be the set of enriched $P$-partitions.
\end{defn} 
\noindent Finally recall the weighted variants of posets introduced in \cite{GriVas21}. 
\begin{defn}[\cite{GriVas21}]
A \emph{labelled weighted poset} is a triple $P = ([n],<_P,\epsilon)$ where $([n], <_P)$ is a labelled poset and $\epsilon : [n]\longrightarrow \PP$ is a map (called the \emph{weight function}).
\end{defn}
\noindent Each node of a labelled weighted poset is marked with its label and weight (Figure \ref{fig : poset}).
\begin{figure}[htbp]
\begin{center}
\begin{tikzcd}[row sep = small]
                                              &  & {2,\ \epsilon(2) = 5} &  &                                  &  &                       \\
{3,\ \epsilon(3) = 2} \arrow[rru,very thick] &  &                                   &  & {1,\ \epsilon(1) = 1} \arrow[llu,very thick] \arrow[rr,very thick] &  & {4,\ \epsilon(4) = 2} \\
                                              &  & {5,\ \epsilon(5)=2} \arrow[llu,very thick]  \arrow[rru,very thick]   &  &                                  &  &                      
\end{tikzcd}
\end{center}
\caption{A $5$-vertex labelled weighted poset. Arrows show the covering relations.}
 \label{fig : poset}
 \end{figure}
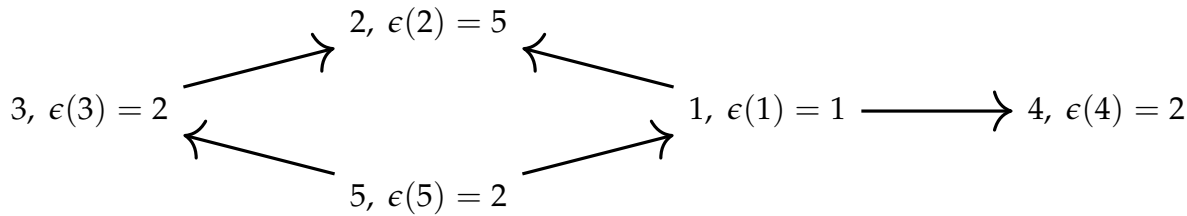
 %
\subsection{Quasisymmetric functions}
\label{sec : QSym}
Consider the set of indeterminates $X = \left\{x_1,x_2,x_3,\ldots\right\}$, the ring $\kk \left[\left[ X \right]\right]$ of formal power series on $X$ where $\kk$ is a commutative ring, and let $\mathcal{Z} \in \{\PP, \PP^{\pm}\}$. Given a labelled weighted poset $([n], <_P, \epsilon)$, define its generating function $\Gamma_\mathcal{Z}([n], <_P, \epsilon) \in \kk \left[\left[ X \right]\right]$  by
\begin{equation}
\label{eq : weightGamma}
\Gamma_\mathcal{Z}([n], <_P, \epsilon) = \sum_{f \in \mathcal{L}_\mathcal{Z}([n], <_P)}\ \ \prod_{1\leq i \leq n}x^{\epsilon(i)}_{|f(i)|},
\end{equation} 
where $|f(i)| = -f(i)$ (resp. $= f(i)$) for $f(i) \in -\PP$ (resp. $\PP$).
Let $S_n$ be the symmetric group on $[n]$. Given a \emph{composition}, i.e. a sequence of positive integers $\alpha = (\alpha_1, \alpha_2, \ldots, \alpha_n)$ with $n$ entries, and a permutation $\pi=\pi_1\dots\pi_n$ of $S_n$, we let $P_{\pi,\alpha} = ([n],<_\pi,\alpha)$ be the labelled weighted poset on the set $[n]$, where the order relation $<_\pi$ is such that $\pi_i <_\pi \pi_j$ if and only if $i < j$ and $\alpha$ is the weight function sending the vertex labelled $\pi_i$ to $\alpha_i$ (see Figure \ref{fig : monomial}). For $\mathcal{Z} \in \{\PP, \PP^{\pm}\}$, its generating function $U^\mathcal{Z}_{\pi,\alpha} = \Gamma_\mathcal{Z}([n],<_\pi, \alpha)$ is called the \emph{universal quasisymmetric function} (\cite{GriVas21}) indexed by $\pi$ and $\al$.
\begin{figure}[htbp]
\begin{center}
\begin{tikzcd}[row sep = small]
{\pi_1,\ \alpha_1} \arrow[r, very thick] &
{\pi_2,\ \alpha_2} \arrow[r, very thick] &
\cdots\cdots\cdots \arrow[r, very thick] &
{\pi_n,\ \alpha_n}
\end{tikzcd}
\end{center}
\caption{The labelled weighted poset $P_{\pi,\alpha}$.}
\label{fig : monomial}
 \end{figure}
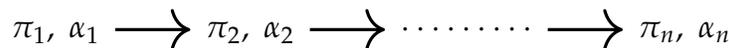
\begin{defn} 
Let $[1^n]$ denote the composition with $n$ entries equal to $1$.  For each $\pi \in S_n$, let $L_{\pi}= U^\mathcal{\PP}_{\pi,[1^n]}$ and $K_{\pi}= U^\mathcal{\PP^\pm}_{\pi,[1^n]}$. The power series $L_\pi$ (resp. $K_\pi$) are  \emph{Gessel's fundamental (resp. Stembridge's peak) quasisymmetric functions} indexed by the permutation $\pi$. 
\end{defn}
The power series $L_\pi$ and $K_\pi$ belong to the subalgebra of $\kk \left[\left[ X \right]\right]$ called the ring of \emph{quasisymmetric functions ($\QSym$)}, i.e. for any strictly increasing sequence 
of indices $i_1 < i_2 <\cdots< i_p$ the coefficient of $x_1^{k_1}x_2^{k_2}\cdots x_p^{k_p}$ is equal to the coefficient of $x_{i_1}^{k_1}x_{i_2}^{k_2}\cdots x_{i_p}^{k_p}$. Furthermore they are related to two major statistics on permutations. Given $\pi \in S_n$, define its \emph{descent set} $\Des(\pi)= \{1\leq i\leq n-1| \pi(i)>\pi(i+1)\}$ and its \emph{peak set} $\Peak(\pi) = \{2\leq i\leq n-1| \pi(i-1)<\pi(i)>\pi(i+1)\}$. The peak set of a permutation is \emph{peak-lacunar}, i.e. it neither contains $1$ nor contains two consecutive integers.
\begin{proposition}[\cite{Ges84,Ste97}] 
\label{prop : LK}
For any permutation $\pi \in S_n$, the fundamental quasisymmetric function $L_\pi$ and the peak quasisymmetric function $K_\pi$ satisfy
\begin{equation*}
L_{\pi} = \sum\limits_{\substack{i_1 \leq \cdots \leq i_n;\\j\in \Des(\pi) \Rightarrow i_j<i_{j+1}}} x_{i_1}x_{i_2} \cdots x_{i_n},~~~~~~~~~~~~~K_{\pi} = \sum\limits_{\substack{i_1 \leq \dots \leq i_n;\\j\in \Peak(\pi) \Rightarrow  i_{j-1}<i_{j+1}}}2^{|\{i_1,i_2,\dots,i_n\}|} x_{i_1}x_{i_2} \cdots x_{i_n}.
\end{equation*}
As a result $L_\pi$ ($K_\pi$) depends only on $n$ and $\Des(\pi)$ ($\Peak(\pi)$) and we may  use set indices and write $L_{n, \Des(\pi)}$ ($K_{n, \Peak(\pi)}$) instead of $L_\pi$ ($K_\pi$). Furthermore $(L_{n,I})_{n\geq0, I\subseteq[n-1]}$ is a basis of $\QSym$ (we assume $[-1] = [0] = \emptyset$), and $(K_{n,I})_{n\geq0, I}$ is a basis of a subalgebra of $\QSym$ called the algebra of peaks when $I$ runs over all peak-lacunar subsets of $[n -1]$ for all integers $n$.
\end{proposition}
\begin{defn}
\label{def : MEE}
Let $id_{n}$ and $\overline{id_{n}}$ denote the permutations in $S_n$ given by  $id_{n} = 1~2~3\dots n$ and $\overline{id_{n}} = n~n-1\dots 1$. Given a composition $\al = (\al_1, \dots, \al_n)$ of $n$ entries, define the \emph{monomial} $M_\al$ (\cite{Ges84}), \emph{essential} $E_\al$ (\cite{Hof15}) and \emph{enriched monomial} $\eta_\al$ (\cite{Hsi07, GriVas21}) quasisymmetric functions
\begin{gather*}
M_{\alpha} = U^{\mathcal{\PP}}_{\overline{id_{n}},\alpha} = \sum_{i_1<\dots<i_n}x_{i_1}^{\al_1}\dots x_{i_n}^{\al_n},
\qquad
E_{\alpha} = U^{\mathcal{\PP}}_{id_{n},\alpha} = \sum_{i_1\leq\dots\leq i_n}x_{i_1}^{\al_1}\dots x_{i_n}^{\al_n},\\
\eta_{\alpha} = U^{\mathcal{\PP^\pm}}_{id_{n},\alpha}=\sum_{i_1\leq\dots\leq i_n}2^{|\{i_1,\dots,i_n\}|}x_{i_1}^{\al_1}\dots x_{i_n}^{\al_n}.
\end{gather*}
\end{defn} 
Compositions $\al = (\al_1,\dots,\al_n)$ such that $\al_1 + \dots + \al_n = s$ are in bijection with subsets of $[s-1]$. For $I \subseteq [s-1]$, we also use the following alternative indexing for monomial, essential and enriched monomials. References to $s$ in indices are removed for clarity. 
\begin{equation*}
M_I = \sum_{\substack{i_1\leq\dots\leq i_s\\ j \in I \Leftrightarrow i_j=i_{j+1}}}x_{i_1}\dots x_{i_s},~~~
E_I = \sum_{\substack{i_1\leq\dots\leq i_s\\ j \in I \Rightarrow i_j=i_{j+1}}}x_{i_1}\dots x_{i_s},~~~
\eta_{I} =\sum_{\substack{i_1\leq\dots\leq i_s\\ j \in I \Rightarrow i_j=i_{j+1}}}2^{|\{i_1,\dots,i_s\}|}x_{i_1}\dots x_{i_s}.
\end{equation*}
\begin{proposition}
\label{prop : LKEE}
Let $s \geq 0$. Let $I$ and $J$ be a subset and a peak-lacunar subset of $[s-1]$. Then,
\begin{equation}
L_{I} = \sum\limits_{U \subseteq I} (-1)^{|U|} E_U,~~~~~~~~~~~~~K_{J} = \sum\limits_{V \subseteq J}(-1)^{|V|} \eta_{(V-1) \cup V},
\end{equation}
where for $V$ peak-lacunar, we set $V-1 = \{v-1| v \in V\}$.
\end{proposition}
\section{A $q$-deformed generating function for $P$-partitions}
Equation (\ref{eq : weightGamma}) and Propositions \ref{prop : LK} and \ref{prop : LKEE} exhibit the strong similarities between enriched and classical $P$-partitions. As we will see, both are special cases of a more general theory. Looking at Equation (\ref{eq : weightGamma}), one may notice that the generating function does not depend on the sign of $f(i)$. Let $\omega$ be the map that sends the element $i$ of a labelled weighted poset $([n],<_{P},\epsilon)$ and an enriched P-partition $f$ to the contributing monomial in $\Gamma$. That is, $\omega(i,f)=x_{|f(i)|}^{\epsilon(i)}$.
As proposed by Stembridge, the value of $\omega$ does not depend on the sign of $f$. We break this
assumption and write for an additional parameter $q$:
\[
\omega(i,f,q)=x_{f(i)}^{\epsilon(i)}\mbox{ if }f(i) \in \PP,~~~~~~
\omega(i,f,q)=qx_{-f(i)}^{\epsilon(i)}\mbox{ if }f(i) \in -\PP.
\]
\begin{defn}
Let $q \in \mathbf{k}$ (the base ring of the power series). The \emph{$q$-generating function} for enriched $P$-partitions on the weighted poset $([n],<_{P},\epsilon)$ is
\begin{equation}
\Gamma_{q}([n],<_{P},\epsilon)=\!\!\!\!\!\!\! \sum_{f\in\mathcal{L}_{\PP^\pm}([n],<_{P})}\prod_{1\leq i\leq n}\omega(i,f,q) =\!\!\!\!\!\!\!\sum_{f\in\mathcal{L}_{\PP^\pm}([n],<_{P})} \prod_{1\leq i\leq n}q^{[f(i)<0]}x_{|f(i)|}^{\epsilon(i)},
\end{equation}
where $[f(i)<0] = 1$ if $f(i)<0$ and $0$ otherwise.
\end{defn}
This definition covers the case of Gessel ($q=0$) with no negative numbers
allowed and the one of Stembridge ($q=1$) where the sign of $f$ is ignored in
the generating function. Define also the \emph{$q$-universal quasisymmetric function}
\begin{equation}
U^q_{\pi,\alpha} = \Gamma_q([n],<_\pi, \alpha).
\end{equation}
\begin{proposition}
\label{prop : Uqq}
Let $q \in \mathbf{k}$, $\pi \in S_n$ and $\alpha = (\alpha_1, \alpha_2, \ldots, \alpha_n)$ be a composition with $n$ entries. Then,
\begin{equation}
U_{\pi,\alpha}^{q}=\sum_{\substack{i_{1}\leq i_{2}\leq
\dots\leq i_{n};\\j\in\Peak(\pi)\Rightarrow 
i_{j-1}<i_{j+1}  }}q^{|\{j\in\Des(\pi
)|i_{j}=i_{j+1}\}|}(q+1)^{|\{i_{1},i_{2},\dots, i_{n}\}|}x_{i_{1}}^{\alpha_{1}%
}x_{i_{2}}^{\alpha_{2}}\dots x_{i_{n}}^{\alpha_{n}}.
\label{eq : Uqq}
\end{equation}
\end{proposition}
\begin{proof}
Let $([n],<_\pi, \alpha)$ be the weighted chain poset associated to $\pi \in S_n$ and to the composition $\al$ with $n$ entries.
Consider an enriched $P$-partition $f \in\mathcal{L}_{\PP^\pm}([n],<_\pi)$ and an $a \in \PP$. All the $i \in [n]$ satisfying $|f(\pi_i)| = a$ form an interval $[j, k] = \{j, j+1, \ldots, k\}$ for some positive integers $j$ and $k$.
By Definition \ref{def : enriched}, we have $[j, k]\cap \Peak(\pi) = \emptyset$. As a result, there exists $l$ such that $\pi_j > \dots > \pi_l<\dots<\pi_k$.
We have $f(\pi_j) = \dots = f(\pi_{l-1}) = -a$, $f(\pi_{l+1}) = \dots = f(\pi_{k}) = a$ and $f(\pi_l) \in \{-a, a\}$. The two contributions in $x_a$ are $$x_a^{\al_j+\al_{j+1}+\cdots+\al_k}[q^{l-j} + q^{l-j+1}] = (q+1)q^{l-j}x_a^{\al_j+\al_{j+1}+\cdots+\al_k}.$$
Note that $l-j = \{i \in \Des(\pi)| |f(\pi_i)| = a\}$ to complete the proof. 
\end{proof}
The nice formula for the product of two generating functions of chain posets extends naturally to this $q$-deformation. Recall the definition of coshuffle from \cite{GriVas21}:
\begin{defn}
Let $\pi \in S_n$ and $\sigma \in S_m$ be two permutations. Let $\alpha$ and $\beta$ be two compositions with $n$ and $m$ entries, respectively. The \emph{coshuffle} of $(\pi,\alpha)$ and $(\sigma,\beta)$, denoted $(\pi,\alpha) \shuffle (\sigma,\beta)$, is the set of pairs $(\tau,\gamma)$ where 
\begin{itemize}
\item $\tau \in S_{n+m}$ is a shuffle of $\pi$ and $n+\sigma = (n+\sigma_1, n+\sigma_2, \ldots, n+\sigma_m)$, and 
\item $\gamma$ is a composition with $n+m$ entries, obtained by shuffling the entries of $\alpha$ and $\beta$ using the \emph{same shuffle} used to build $\tau$ from the letters of $\pi$ and $n+\sigma$. 
\end{itemize}
\begin{example}
$(1\mathbf{3} 2, (2, \mathbf{1}, 2))$ is a coshuffle of $(12,(2,2))$ and $(1,(1))$.
\end{example}
\end{defn}
\begin{proposition}
\label{prop : UUU}
Let $q \in \mathbf{k}$, let $\pi$ and $\sigma$ be two permutations in $S_n$ and $S_m$, and let $\alpha = (\al_1,\dots,\al_n)$ and $\beta = (\be_1,\dots,\be_m)$ be two compositions with $n$ and $m$ entries.
The product of two $q$-universal quasisymmetric functions is given by
\begin{equation}
\label{eq : UU}
U^q_{\pi,\alpha}U^q_{\sigma,\beta}=\sum_{(\tau,\gamma)\in(\pi,\alpha)\shuffle(\sigma,\beta)}U^q_{\tau,\gamma} .
\end{equation}
\end{proposition}
\begin{proof}
The proof is similar to \cite[Thm. 3]{GriVas21}.
\end{proof}
\section{Enriched $q$-monomials}
\subsection{Definition, relation to $q$-universal quasisymmetric functions and product formula}
We introduce a new basis of $\QSym$ that generalises the essential and enriched monomial quasisymmetric functions in Definition \ref{def : MEE}. 
\begin{defn}[Enriched $q$-monomials]
\label{def : EQ}
Let $q \in \mathbf{k}$ and $\al$ be a composition with $n$ entries. The \emph{enriched $q$-monomial} indexed by $\al$ is defined as
\begin{equation}
\label{eq : EU}
\eta^{(q)}_\al = U^q_{id_n,\alpha}.
\end{equation}
\end{defn}
As an immediate consequence of Definition \ref{def : EQ}, one has $\eta^{(0)}_\al = E_\al$ and $\eta^{(1)}_\al = \eta_\al$. 
\begin{proposition}
Let $q \in \kk$, and let $\al = \left(\al_1, \al_2, \ldots, \al_n\right)$ be a composition with $n$ entries. Then,
\begin{equation}
\eta^{(q)}_\al = \sum_{i_1\leq i_2 \leq \dots \leq i_n}(q+1)^{|\{i_{1},i_{2},\dots, i_{n}\}|}x_{i_{1}}^{\alpha_{1}}x_{i_{2}}^{\alpha_{2}}\dots x_{i_{n}}^{\alpha_{n}}.
\end{equation}
\end{proposition}
\begin{proof}
This is a direct consequence of Proposition \ref{prop : Uqq}.
\end{proof}
Interestingly, one may express general $q$-universal quasisymmetric functions in terms of the $\eta^{(q)}_\al$. To state this result we need the following definition.
\begin{defn}
Let $\al = (\al_1, \dots, \al_n)$ be a composition with $n$ entries.
For any integer $1 \leq i \leq n-1$, we let $\al^{\downarrow i}$ denote the following composition with $n-1$ entries:
\begin{gather*}
\al^{\downarrow i} = (\al_1, \dots, \al_{i-1},{\al_i+\al_{i+1}},\al_{i+2},\dots ,\al_n) .
\end{gather*}  
Furthermore, for any subset $I\subseteq [n-1]$, we set
\[
\al^{\downarrow I}
= \left( \left( \cdots \left( \al^{\downarrow i_k} \right) \cdots \right)^{\downarrow i_2} \right)^{\downarrow i_1},
\]
where $i_1, i_2, \ldots, i_k$
are the elements of $I$
in increasing order.
Finally, if $I$ and $J$ are two subsets of $[n-1]$, with $J$ being peak-lacunar, then we set $\al^{\downarrow I \downarrow\downarrow J} = \al^{\downarrow K}$, where $K = I \cup J \cup (J-1)$.
\end{defn}
\begin{theorem}
\label{thm : UE}
Let $\pi \in S_n$ be a permutation and $\al$ be a composition with $n$ entries. The $q$-universal quasisymmetric function $U^q_{\pi,\alpha}$ may be expressed as a combination of the enriched $q$-monomials: 
\begin{equation}
\label{eq : UE}
U^q_{\pi,\alpha} = \sum_{\substack{I \subseteq \Des(\pi)\\
J \subseteq \Peak(\pi)\\
I\cap J = \emptyset}} (-q)^{|J|}(q-1)^{|I|}\eta^{(q)}_{\alpha^{\downarrow I\downarrow\downarrow J}}.
\end{equation}
\end{theorem}
\begin{proof}
For any subset $V \subseteq [n-1]$, set $\overline{V} = [n-1]\setminus V$. Then, \eqref{eq : Uqq} becomes\footnote{We understand $i_{j-1}$ to be $0$ whenever $j = 1$.}
\begin{align*}
U_{\pi,\alpha}^{q}&=\sum_{K \subseteq \Des(\pi)}\sum_{\substack{i_{1}\leq i_{2}\leq
\dots\leq i_{n}\\
j \in \Des(\pi)\setminus K\Rightarrow i_{j-1}\leq i_{j} < i_{j+1}\\
j \in K\cap \Peak(\pi) \Rightarrow i_{j-1}< i_{j} = i_{j+1}\\
j \in K\cap \overline{\Peak(\pi)} \Rightarrow i_{j-1} \leq i_{j} = i_{j+1}\\
  }}q^{|K|}(q+1)^{|\{i_{1},i_{2},\dots, i_{n}\}|}x_{i_{1}}^{\alpha_{1}%
}x_{i_{2}}^{\alpha_{2}}\dots x_{i_{n}}^{\alpha_{n}} \\
&=\sum_{
\substack{
K \subseteq \Des(\pi)\\
U \subseteq \Des(\pi)\setminus K\\
J \subseteq K\cap \Peak(\pi)\\
}
}
q^{|K|}(-1)^{|U|+|J|}
 \!\!\!\!\! \!\!\!\!\! \!\!\!\!\! \!\!\!\!\! \sum_{\substack{i_{1}\leq i_{2}\leq
\dots\leq i_{n}\\
j \in U \cup K\cap \overline{\Peak(\pi)} \cup K\cap \Peak(\pi) \setminus J \Rightarrow i_{j-1}\leq i_{j} = i_{j+1}\\
j \in J \Rightarrow i_{j-1}= i_{j} = i_{j+1}
 }
 }
 \!\!\!\!\! (q+1)^{|\{i_{1},i_{2},\dots, i_{n}\}|}x_{i_{1}}^{\alpha_{1}%
}x_{i_{2}}^{\alpha_{2}}\dots x_{i_{n}}^{\alpha_{n}} \\
&=\sum_{
\substack{
K \subseteq \Des(\pi)\\
U \subseteq \Des(\pi)\setminus K\\
J \subseteq K\cap \Peak(\pi)\\
}
}
q^{|K|}(-1)^{|U|+|J|}
 \!\!\!\!\! \!\!\!\!\! \!\!\!\!\! \!\!\!\!\! \sum_{\substack{i_{1}\leq i_{2}\leq
\dots\leq i_{n}\\
j \in U \cup K \setminus J \Rightarrow i_{j-1}\leq i_{j} = i_{j+1}\\
j \in J \Rightarrow i_{j-1}= i_{j} = i_{j+1}
 }
 }
 \!\!\!\!\! (q+1)^{|\{i_{1},i_{2},\dots, i_{n}\}|}x_{i_{1}}^{\alpha_{1}%
}x_{i_{2}}^{\alpha_{2}}\dots x_{i_{n}}^{\alpha_{n}}.
\end{align*}
If we set $I=U \cup K \setminus J$ and $U' = I \setminus U = K \setminus J$, then $|U'| = |K| - |J|$ and $I \subseteq \Des(\pi)\setminus J$.
Thus, the above computation becomes
\begin{align*}
U_{\pi,\alpha}^{q}&=\sum_{
\substack{
U' \subseteq I\\
I \subseteq \Des(\pi)\\
J \subseteq \Peak(\pi)\\
I\cap J = \emptyset
}
}
q^{|U'| + |J|}(-1)^{|U'| + |I| + |J|}
 \!\!\!\!\! \!\!\!\!\! \sum_{\substack{i_{1}\leq i_{2}\leq
\dots\leq i_{n}\\
j \in I \Rightarrow i_{j-1}\leq i_{j} = i_{j+1}\\
j \in J \Rightarrow i_{j-1}= i_{j} = i_{j+1}
 }
 }
 \!\!\!\!\! (q+1)^{|\{i_{1},i_{2},\dots, i_{n}\}|}x_{i_{1}}^{\alpha_{1}%
}x_{i_{2}}^{\alpha_{2}}\dots x_{i_{n}}^{\alpha_{n}}.
\end{align*}
Summing over $U'$ yields formula (\ref{eq : UE}).
\end{proof}
\begin{corollary}
\label{cor : EEE}
Let $\al=(\al_1,\dots,\al_n)$ and $\beta=(\be_1,\dots,\be_m)$ be two compositions. Let $\al\shuffle\be$ be the multiset of compositions obtained by shuffling $\al$ and $\be$. As in \cite{GriVas21}, given $\gamma \in  \al\shuffle\be$, let $S_\beta(\gamma)$ be the set of the positions of the entries of $\beta$ in $\gamma$. Set furthermore $S_\beta(\gamma)-1 =\{i-1|i \in S_\beta(\gamma)\}$ and $T_\beta(\gamma) = S_\beta\tup{\gamma} \setminus \tup{S_\beta\tup{\gamma} - 1}$. Then,
\begin{equation}
\eta_{\alpha}^{(q)}\eta_{\beta}^{(q)}=\sum_{\substack{\gamma\in\alpha
\shuffle\beta;\\I \subseteq T_{\beta}(\gamma) \setminus \set{n+m} ;\\ J\subseteq T_{\beta}(\gamma) \setminus \set{1, n+m}  ;\\I\cap J=\emptyset
}}(q-1)^{|I|}(-q)^{|J|}\eta_{\gamma^{\downarrow I\downarrow\downarrow J}}%
^{(q)}.\label{eq : EEE}%
\end{equation}
\end{corollary}
\begin{proof}
Corollary \ref{cor : EEE} is a consequence of Theorem \ref{thm : UE}, Equation (\ref{eq : EU}) and Proposition \ref{prop : UUU}.
\end{proof}
\subsection{Relation to the monomial and fundamental bases}
We consider the alternative indexing with sets proposed at the end of Section \ref{sec : QSym}. Given a set of positive integers $I \subseteq [s-1]$, the enriched $q$-monomial may be written as
\begin{equation*}
\eta^{(q)}_{I} =\sum_{\substack{i_1\leq\dots\leq i_s\\ j \in I \Rightarrow i_j=i_{j+1}}}(q+1)^{|\{i_1,\dots,i_s\}|}x_{i_1}\dots x_{i_s}.
\end{equation*}
\begin{proposition}
Let $I \subseteq [s-1]$ be a set of positive integers. One has
\begin{equation}
\label{eq : EM}
\eta^{(q)}_I = \sum_{I \subseteq J}(q+1)^{s - |J|}M_J.
\end{equation}
\end{proposition}
\begin{theorem}
Let $q \in \mathbf{k}$ be such that $q+1$ is invertible. The family of enriched $q$-monomial quasisymmetric functions $\left(\eta^{(q)}_{s,I} \right )_{s\geq0, I\subseteq[s-1]}$ is a basis of $\QSym$. Furthermore
\begin{equation}
\label{eq : ME}
(q+1)^{s - |J|}M_J = \sum_{J \subseteq I}(-1)^{|I\setminus J|}\eta^{(q)}_I.
\end{equation}
\end{theorem}
\begin{proof}
Follows from Equation (\ref{eq : EM}) by Möbius inversion. 
\end{proof}
We develop further the properties of the enriched $q$-monomial basis of $\QSym$.
\begin{proposition}
\label{prop : EL}
Let $s$ be a positive integer and $I \subseteq [s-1]$. One may expand the enriched $q$-monomials in the fundamental basis as
\begin{equation}
\label{eq : EL}
\eta^{(q)}_I = (q+1) \sum_{J \subseteq [s-1]}(-1)^{|J|}(-q)^{|J\setminus I|}L_J.
\end{equation} 
\end{proposition}
\begin{proof}
The expression above is a consequence of Equation (\ref{eq : EM}) and the expansion of monomial quasisymmetric functions in the fundamental basis (see e.g. \cite{Ges84}). 
\end{proof}
\begin{proposition}
Let $s$ be a positive integer, $J \subseteq [s-1]$ and let $q \in \mathbf{k}$. Then,
\begin{equation}
\label{eq : LE}
(q+1)^s L_J = \sum_{I \subseteq [s-1]}(-1)^{|I|}(-q)^{|I\setminus J|}\eta^{(q)}_I.
\end{equation} 
\end{proposition}
Equations (\ref{eq : EL}) and (\ref{eq : LE}) expand the fundamental and enriched $q$-monomial bases in terms of one another, and thus suggest a duality relation between the two. Let  $\QSym_s$ be the vector subspace of $\QSym$ containing the homogeneous quasisymmetric functions of degree $s$. Define 
$f:\QSym_{s}\rightarrow \QSym_{s}$ as the $\mathbf{k}$-linear map that sends each $L_I$ to
$\eta^{(q)}_I$ for $I \subseteq [s-1]$. Then $f^{2}$ is a scaling by $(q+1)^{s+1}$
(that is, $f^{2}=(q+1)^{s+1}\operatorname*{id}$). Moreover,
\[
f\left(  M_I\right)  = (q+1)^{|I| + 1 } M_{[s-1]\setminus I}
\qquad \text{for any $I \subseteq [s-1]$.}
\]
\subsection{Antipode}
\label{sec : anti}
For an integer $s$ and a subset $I \subseteq [s-1]$, we set $s-I = \{s-i|i \in I\}$. The \emph{antipode} of $\QSym$ (see \cite[Chapter 5]{GriRei20}) can be defined as the unique $\mathbf{k}$-linear map $S : \QSym \to \QSym$  that satisfies
\[
S\left(  M_{I}\right)  =\left(  -1\right)  ^{s-|I|}\sum_{(s-I) \subseteq J}M_{J}.
\]

\begin{proposition} Assume that $q$ is invertible in $\mathbf{k}$,
and let $p=\dfrac{1}{q}$. Then, for $I \subseteq [s-1]$,
\begin{equation}
\label{eq : anti}
S\left(  \eta_{I}^{\left( q\right)  }\right)  =\left(  -q\right)
^{s - |I|  }\eta_{s-I}^{\left( p\right)  }\ .
\end{equation}
\end{proposition}
\begin{proof}
This can be derived from Equation (\ref{eq : EL}).
\end{proof}
\section{A $q$-interpolation between Gessel and Stembridge quasisymmetric functions}
\subsection{$q$-fundamental quasisymmetric functions}
We introduce a new family of quasisymmetric functions that interpolate between Gessel's fundamental and Stembridge peak quasisymmetric functions and show that it is a basis of $\QSym$ in all but the Stembridge case.
\begin{defn}[$q$-fundamental quasisymmetric functions]
Let $\pi$ be a permutation in $S_n$ and $q \in \mathbf{k}$. Define the \emph{$q$-fundamental quasisymmetric function} indexed by $\Des(\pi)$ as
\begin{equation}
L_{n, \Des(\pi)}^{(q)} = U^{q}_{\pi, [1^{n}]}.
\end{equation}
\end{defn}
Let $I$ be a subset of $[n-1]$. Set $I+1 = \{i+1|i \in I\}$, and let $\Peak(I) = I\setminus (I+1)
\setminus \{1\}$ the peak-lacunar subset obtained from $I$ (so $\Peak(I) = \Peak(\pi)$ for every $\pi \in S_n$ satisfying $\Des(\pi) = I$). One recovers immediately that for $q=0$, $L_{n, I}^{(0)} = L_{n, I}$ is the Gessel fundamental quasisymmetric
function indexed by the set $I$. For $q=1$, $L_{n, I}^{(1)} = K_{n, \Peak(I)}$ is the Stembridge peak function indexed by the relevant peak-lacunar set. In the sequel we remove the reference to $n$ in indices when it is clear from context. 
Proposition \ref{prop : LKEE} admits a nice generalisation to this $q$-deformation.
\begin{theorem} Let $I \subseteq[n-1]$ and $q \in \mathbf{k}$. The $q$-fundamental quasisymmetric functions may be expressed in the enriched $q$-monomial basis as
\begin{equation}
\label{eq : LEq}
L_{I}^{(q)}
= \sum_{\substack{J \subseteq I\\K \subseteq \Peak(I)\\J \cap K = \emptyset}}
(-q)^{|K|}(q-1)^{|J|}\eta^{(q)}_{J\cup (K-1) \cup K} \ .
\end{equation}
\end{theorem}
\begin{proof}
This a consequence of Equation (\ref{eq : UE}).
\end{proof}
\begin{proposition}
Recall the antipode $S$ of Section \ref{sec : anti}. Let $q \in \mathbf{k}$ be invertible, and set $p=\frac{1}{q}$. Let $I \subseteq [n-1]$, and set $n - I = \left\{n - i \mid i \in I\right\}$. Then,
 \begin{equation}
S(L_{I}^{(q)}) = (-q)^n L_{n-I}^{(p)}.
\end{equation}
\end{proposition}
\begin{proof}
This a consequence of Equations  (\ref{eq : LEq}) and (\ref{eq : anti}).
\end{proof}
To know whether $(L_{n, I}^{(q)})_{n\geq0, I\subseteq[n-1]}$ is a basis of $\QSym$ for some value of $q$ appears as a natural question. For example, for $n=3$, we can invert Equation (\ref{eq : LEq}) as follows:

\begin{itemize}
\item $\eta^{(q)}_{\emptyset} = L^{(q)}_{\emptyset}$;

\item $(q-1)\eta^{(q)}_{\{1\}} = L^{(q)}_{\{1\}} - L^{(q)}_{\emptyset}$;

\item $(q-1)\eta^{(q)}_{\{2\}} = \frac{(q-1)^{2}}{(q-1)^{2} + q}(L^{(q)}_{\{2\}} -
L^{(q)}_{\emptyset}) + \frac{q}{(q-1)^{2} + q}(L^{(q)}_{\{1,2\}} -
L^{(q)}_{\{1\}})$;

\item $\eta_{\{1,2\}}^{(q)}=\frac{1}{(q-1)^{2}+q}\left(L_{\{1,2\}}^{(q)}-L_{\{2\}}^{(q)}-L_{\{1\}}^{(q)}+L_{\emptyset}^{(q)}\right)$.
\end{itemize}
We see that except for the case of Stembridge $q=1$ (and the degenerate case $q=-1$), $\left (L_{2, I}^{(q)}\right)_{I\subseteq [2]}$ seems to be a basis of $\QSym$. We state one of our main theorems:
\begin{theorem}\label{thm : basis} Let $\mathbf{k}$ be the set $\RR$ of real numbers. The family of $q$-fundamental quasisymmetric functions $(L_{n, I}^{(q)})_{n\geq0, I\subseteq[n-1]}$ is a basis of $\QSym$ for $q \notin \{-1, 1\}$.
\end{theorem}
\begin{remark}
We set $\mathbf{k} = \RR$ for the sake of simplicity. For a more general field, $(L_{n, I}^{(q)})_{n, I\subseteq[n-1]}$ is a basis if and only if $q \notin \{\rho| \rho^k=1 \mbox{ for some integer }k > 0\}$.
\end{remark}
\subsection{Proof of Theorem \ref{thm : basis} }
To prove Theorem \ref{thm : basis} we characterise the transition matrix between the $q$-fundamental and enriched $q$-monomial quasisymmetric functions and show it is invertible for $q\neq -1,1$.
\begin{defn}
\label{def : Bn}
Let $B_n$ be the transition matrix between $(L_{I}^{(q)})_{I \subseteq [n-1]}$ and $(\eta^{(q)}_{J})_{J \subseteq [n-1]}$ with coefficients given by Equation (\ref{eq : LEq}). Columns and rows are indexed by subsets $I$ of $[n-1]$ sorted in reverse lexicographic order. A subset $I$ is before subset $J$ iff the word obtained by writing the elements of $I$ in decreasing order is before the word obtained from $J$ for the lexicographic order. 
\end{defn}
\begin{example} 
For $n=4$, let us show the transition matrix $B_{4}$ between $(L_{I}^{(q)})_{I \subseteq [3]}$ and $(\eta^{(q)}_{J})_{J \subseteq [3]}$. The entry at row index $I$ and column index $J$ is the coefficient in $\eta^{(q)}_J$ of $L^{(q)}_I$ in Equation (\ref{eq : LEq}).
\begin{equation*}
B_{4} = 
\begin{array}{c|cccccccc}
 & \emptyset & \{1\} & \{2\} &  \{2, 1\} & \{3\} &  \{3, 1\}& \{3, 2\} &  \{3, 2, 1\}\\
 \hline
 \emptyset & 1 & 0 & 0 & 0 & 0 & 0 & 0 & 0\\
\{1\} & 1 & q-1 & 0 & 0 & 0 & 0 & 0 & 0\\
\{2\} & 1 & 0 & q-1 & -q & 0 & 0 & 0 & 0\\
\{2, 1\} & 1 & q-1 & q-1 & (q-1)^2 & 0 & 0 & 0 & 0\\
\{3\} & 1 & 0 & 0 & 0 & q-1 & 0 & -q & 0\\
\{3, 1\} & 1 & q-1 & 0 & 0 & q-1 & (q-1)^2 & -q & -q(q-1)\\
\{3, 2\} & 1 & 0 & q-1 & -q & q-1 & 0 & (q-1)^2 & -q(q-1)\\
\{3,2,1\} & 1 & q-1 & q-1 & (q-1)^2 & q-1 & (q-1)^2 & (q-1)^2 & (q-1)^3\\
\end{array}
\end{equation*}
\end{example}
Using Definition \ref{def : Bn} and Equation (\ref{eq : LEq}), one can deduce the following lemmas.
\begin{lem}
\label{lem : Bn}
The matrix $B_n$ is block triangular. To be more specific:

For each $k \in [n]$, let $A_k$ denote the transition matrix from $(L_{I}^{(q)})_{I \subseteq [n-1],~  \max(I) = k-1}$ to $(\eta^{(q)}_{J})_{J \subseteq [n-1],~\max(J) = k-1}$ (where $\max\varnothing := 0$); this actually does not depend on $n$.
Note that $A_k$ is a $2^{k-2} \times 2^{k-2}$-matrix if $k \geq 2$, whereas $A_1$ is a $1 \times 1$-matrix.
We have
\begin{equation*}
B_{n} = 
\begin{pmatrix}
 A_1& 0& 0& \hdots& 0\\
 *&A_2& 0& \hdots& 0\\
 *&*&A_3&\hdots& 0\\
 *&*&*&\ddots&0\\
 *&*&*&*&A_n\\
 \end{pmatrix}.
\end{equation*}
\end{lem}
\begin{lem}
\label{lem : rec}
The matrices $\left(B_n\right)_n$ and $\left (A_n\right)_n$ satisfy the following recurrence relations (for $n \geq 1$ and $n \geq 2$, respectively):
\begin{equation*}
B_{n} = 
\begin{pmatrix}
 B_{n-1}& 0\\
B_{n-1}&A_n\\
 \end{pmatrix},
 \qquad
A_{n} = 
\begin{pmatrix}
 (q-1)B_{n-2}& -qB_{n-2}\\
(q-1)B_{n-2}&(q-1)A_{n-1}\\
 \end{pmatrix}.
\end{equation*}
\end{lem}
Thanks to Lemmas \ref{lem : Bn} and \ref{lem : rec}, we are ready to state and show the main proposition of this section and prove Theorem \ref{thm : basis}.
\begin{proposition} The matrix $B_n$ is invertible for $q \neq 1$.
\end{proposition}
\begin{proof}
For any square matrix $M$, let $|M|$ denote its determinant. We want to show that for all $n$, $|B_n| \neq 0$ or equivalently that $|A_n| \neq 0$.
To this end we compute for any rational functions in $q$, $\alpha$ and $\beta$:
\begin{equation}
\label{eq : recdet}
|\alpha A_n + \beta B_{n-1}| = ((q-1)\alpha + \beta)|B_{n-2}||((q-1)\alpha + \beta)A_{n-1} + q\alpha B_{n-2}|.
\end{equation}
Equation (\ref{eq : recdet}) exhibits a recurrence relation on the determinants that we solve by defining the sequence of coefficients:
\begin{equation*}
\begin{pmatrix}
 \alpha_0\\
\beta_0\\
 \end{pmatrix} = \begin{pmatrix}
 1\\
0\\
 \end{pmatrix},~~~~~
 \begin{pmatrix}
 \alpha_i\\
\beta_i\\
 \end{pmatrix} =\begin{pmatrix}
 q-1& 1\\
q&0\\
 \end{pmatrix} \begin{pmatrix}
 \alpha_{i-1}\\
\beta_{i-1}\\
 \end{pmatrix} = \begin{pmatrix}
 q-1& 1\\
q&0\\
 \end{pmatrix}^i\begin{pmatrix}
 \alpha_{0}\\
\beta_{0}\\
 \end{pmatrix}.
 \end{equation*}
 We have:
 $$|A_n| = \left [\prod_{i=0}^{n-3}|B_{n-2-i}|\begin{pmatrix}
 q-1 & 1\\
 \end{pmatrix} \begin{pmatrix}
 \alpha_{i}\\
\beta_{i}\\
 \end{pmatrix}\right ]\left |\begin{pmatrix}
 A_2 & B_1\\
 \end{pmatrix} \begin{pmatrix}
 \alpha_{n-2}\\
\beta_{n-2}\\
 \end{pmatrix}\right|.$$
 But $A_2 = (q-1)$, $B_1 = (1)$ and one may compute that (left to the reader):
 $$\begin{pmatrix}
 q-1 & 1\\
 \end{pmatrix} \begin{pmatrix}
 \alpha_{i}\\
\beta_{i}\\
 \end{pmatrix} = \frac{1}{q+1}\left ( q^{i+2} - (-1)^{i+2}\right) = (-1)^{(i+1)}[i+2]_{-q},$$
 where for integer $p$, $[p]_q$ is the \emph{$q$-number}, $[p]_q = 1 + q + q^2 + \dots + q^{p-1}$. Define the \emph{$q$-factorial} $[p]_q! = [1]_q\cdot[2]_q\cdots[p]_q$. We find
 $$
 |A_n| = (-1)^{n(n-1)/2}[n]_{-q}!\prod_{i=1}^{n-2}|B_{i}|.
 $$
 Then, notice that $[n]_{-q}!$ is $0$ if and only if $q=1$ and $n>1$ (when $q$ runs over real numbers). Finish the proof with a simple recurrence argument on $|B_{i}|$.
 \end{proof}

\printbibliography

@article{Sta72,
	Author = {Stanley, R.},
	Journal = {Memoirs of the American Math. Soc.},
	Title = {Ordered structures and partitions},
	Volume = {119},
	Year = {1972}}

@article{GriVas21,
	Author = {Grinberg, D. and Vassilieva, E.A.},
	Journal = {S{\'e}min. Loth. de Comb.},
	Title = {Weighted posets and the enriched monomial basis of QSym},
	Volume = {85B (FPSAC 2021)},
	Year = {2021},
    Note = {arXiv:2202.04720v1}}

@article{Hof15,
	Author = {M. E. Hoffman},
	Journal = {Kyushu J. Math.},
	Pages = {345--366},
	Title = {Quasi-symmetric functions and mod p multiple harmonic sums},
	Volume = {69},
	Year = {2015},
	Note = {arXiv:math/0401319v3}}

@misc{GriRei20,
	Author = {Grinberg, D. and Reiner, V.},
	Note = {arXiv:1409.8356v7},
	Title = {{H}opf Algebras in Combinatorics},
	Url = {http://www.cip.ifi.lmu.de/~grinberg/algebra/HopfComb-sols.pdf},
	Year = {2020}}

@unpublished{Hsi07,
	Author = {Hsiao, S. K.},
	Title = {Structure of the peak Hopf algebra of quasisymmetric functions},
	Year = {2007}}

@article{Ste97,
	Author = {Stembridge, J.},
	Journal = {Trans. Amer. Math. Soc.},
	Number = {2},
	Pages = {763--788},
	Title = {Enriched $P$-partitions.},
	Volume = {349},
	Year = {1997}}

@book{Sta01,
	Author = {Stanley, R.},
	Publisher = {Cambridge University Press},
	Title = {Enumerative combinatorics},
	Volume = {2},
	Year = {2001}}

@article{Ges84,
	Author = {Gessel, I.},
	Journal = {Contemporary Mathematics},
	Pages = {289--317},
	Title = {Multipartite {P}-partitions and inner products of skew {S}chur functions},
	Volume = {34},
	Year = {1984}}

\end{document}